\documentclass{amsart}

\usepackage{amssymb}
\usepackage{enumerate}
\usepackage{amsmath}
\usepackage{amsthm}
\usepackage{cancel}
\usepackage{tikz} 
\usepackage{float}
\usepackage{mathtools}
\usepackage{graphicx}
\usepackage{biblatex}

\usepackage{pgf,tikz}\usepackage{mathrsfs}\usetikzlibrary{arrows}

\numberwithin{equation}{section}

\newtheorem{thm}{Theorem}[section]

\newtheorem{lem}[thm]{Lemma}

\theoremstyle{definition}

\newcommand{\sn}{S_n}
\newcommand{\bn}{B_n}
\newcommand{\dn}{D_n}

\newcommand{\ol}[1]{\overline{#1}}

\newcommand{\mylabel}[2]{#2\def\@currentlabel{#2}\label{#1}}

\begin{document}
\definecolor{ffqqqq}{rgb}{1.,0.,0.}
\definecolor{qqqqff}{rgb}{0.,0.,1.}
\definecolor{ududff}{rgb}{0.30196078431372547,0.30196078431372547,1.}
\definecolor{ffqqtt}{rgb}{1.,0.,0.2}
\definecolor{zzttqq}{rgb}{0.6,0.2,0.}
\definecolor{uuuuuu}{rgb}{0.26666666666666666,0.26666666666666666,0.26666666666666666}

\author{Sa\'ul A. Blanco and Charles Buehrle}
\title{Presentations of Coxeter groups of type $A$, $B$, and $D$ using prefix-reversal generators}
\address{Department of Computer Science, Indiana University, Bloomington, IN 47408}
\email{sblancor@indiana.edu}
\address{ Department of Mathematics, Physics, and Computer Studies, Notre Dame of Maryland University, Baltimore, MD 21210}
\email{cbuehrle@ndm.edu}



\date{\today}


\begin{abstract}
Here we provide three new presentations of Coxeter groups type $A$, $B$, and $D$ using prefix reversals (pancake flips) as generators. We prove these presentations are of their respective groups by using Tietze transformations on the presentations to recover the well known presentations with generators that are adjacent transpositions. 
\end{abstract}

\maketitle

\section{Introduction}\label{s:intro}

Obtaining an \emph{abstract definition} or \emph{presentation} of a group in terms of generators and relations on those generators has a long history. For example, see Burnside, Carmichael, and Coxeter and Moser \cite{Burnside, Carmichael, CoxMoser}. In these historic examples, the choice of generators and their defining relations are more flexible. Here our motivation is to begin with a particular set of generators, the \emph{prefix reversals} or \emph{pancake flips}, and determine the required set of relations that will present the intended group. The groups that we provide presentations for herein are particular types of \emph{Coxeter groups}, specifically those of type $A$, $B$, and $D$.   

\emph{Prefix-reversal sorting} or \emph{pancake sorting} was first introduced in the Monthly by Jacob E. Goodman (under the pseudonym Harry Dweighter) \cite{Dweighter75}. The setup is that a permutation of length $n$ is to be sorted using only reversals of the first $k$ elements, where $k$ may be any integer between 2 and $n$. The natural setting for this type of sorting is in the \emph{symmetric group of degree $n$}, denoted by $\sn$ or $A_{n-1}$ in the context of Coxeter groups. 

The first significant bounds for the maximum number of prefix reversals required to sort was given in Gates and Papadimitriou \cite{GatesPapa}, which were later improved by Chitturi et al.~in~\cite{Chitt}. In their work, Gates and Papadimitriou also introduced the \emph{burnt-pancake problem}. This variant introduces an ``orientation'' (or sign) to the elements in the permutation. In the context of Coxeter groups, these \emph{signed prefix reversals} generate the \emph{hyperoctahedral group} or \emph{signed symmetric group of degree $n$}, denoted by $B_n$. An index 2 subgroup of these groups are those signed permutations with an even number of negative elements. This subgroup, which is the symmetry group of the $n$-dimensional demihypercube, is denoted by $D_n$. Although they are not considered in any variant of pancake sorting we provide a prefix-reversal presentation of these groups, out of convenience.

This work is, in a way, a continuation of work begun my the authors in \cite{BBrelations}. In that work, ``Coxeter-like'' relations on the prefix-reversal generators of $S_n$ and $B_n$. That type of relation was not sufficient to provide a presentation of either group, alone. This work serves as a culmination of the effort, where a full collection of relations is provided. For small values of $n$, cardinality checks were performed in GAP4~\cite{GAP4} to verify the group presentations (up to $n=12$ for $S_n$ and up to $n=8$ for $B_n$ and $D_n$). 

\textbf{Main Results.} The main results of this paper are a presentation for the Coxeter groups $A_n,B_n$, and $D_n$ in terms of prefix reversals as generators. These presentations can be found in Theorem~\ref{t:snpres},~\ref{t:bnpres}, and ~\ref{t:dnpres}. 

There are some applications to the pancake problem that one could explore. For example, one could use the Knuth-Bendix algorithm to create a confluent rewriting system that could be of used in reducing randomly generated permutations expressed as words in prefix reversals to find some probabilistic predictions of (burnt) pancake numbers that are presently unknown.

\section{Terminology and Notation}\label{s:notation}


A \emph{(finite) presentation} of a discrete group $G$ is given by $\langle S \mid R\rangle$. The set of \emph{generators} $S$ is a finite list of elements of $G$ such that any element of $G$ is expressible as a finite product of elements of $S$. The set $R$ is a collection of \emph{defining relators} on the elements $S$ which are equal to the \emph{identity element} $e$ of $G$. The group $G$ is then isomorphic to the quotient of the free group on $S$ and the normal closure of $R$. Similar to relators are \emph{relations} which are equalities of elements in the group. Note that for $v_1,v_2 \in G$ that $v_1=v_2$ is a relation if $v_1v_2^{-1}$ is a relator. In some places we strike-through pairs of subwords that form a relator, e.g.~$uv \in R$ so $\cancel{u}\cancel{v}w = w$. 

A word $w$ is \emph{derivable} from a set of relators $R=\{R_i\}$ if a finite number of insertions of any $R_i$ or $R_i^{-1}$ between any symbols of $w$ or a finite number of deletions of any $R_i$ or $R_i^{-1}$ in $w$ change $w$ to the empty word. The \emph{Tietze transformations} are a collection of transformations that do not affect the group isomorphism class of a given group presentation. The Tietze transformations as stated in Magnus, Karrass, and Solitar \cite{Magnus} are given below: 
\begin{quote}
    Given a presentation of $G = \langle a, b, c, \ldots \mid P, Q, R,\ldots\rangle$, any other presentation can be collected by a repeat application of: 
    \begin{enumerate}
     \item[\mylabel{T1}{T1}] If words $S,T,U,\ldots$ are derivable from $P,Q,R,\ldots$, then add $S,T,U,\ldots$ to the defining relators.
     \item[\mylabel{T2}{T2}] If some relators, say $S,T,U,\ldots$, listed among $P,Q,R,\ldots$, are derivable from the others, delete $S,T,U,\ldots$ from the relators.
     \item[\mylabel{T3}{T3}] If $K,M,N,\ldots$ are any words in $a,b,c,\ldots$, then adjoin the symbols $x,y,z,\ldots$ to generators and adjoin $x=K,y=M,z=N,\ldots$ to relators.
     \item[\mylabel{T4}{T4}] If some relators take the form $p=V, q=W, \ldots$ where $p,q,\ldots$ are generators and $V,W,\ldots$ are words in other generators than $p,q,\ldots$, then delete $p,q,\ldots$ from generators, delete $p=V,q=W,\ldots$ from relators, and replace $p,q,\ldots$ by $V,W,\ldots$ respectively, in relators.    
    \end{enumerate}
\end{quote}

Throughout this article we adopt the notation of \emph{integer intervals} to represent the set $[i,j] = \{ i, i+1, i+2, \ldots, j\}$ for any integers $i,j$ with $i<j$. If $i=1$, then the interval may be written as $[j] = \{1, 2, 3, \ldots, j\}$. We also utilize \emph{exponent notation} to represent repeated multiplication by one or several group elements, e.g. $(abc)^3 = (abc)(abc)(abc)$. 

We follow the standard combinatorial description of Coxeter groups (see Bj\"{o}rner and Brenti \cite{BjornerBrenti}). Presentations for \emph{Coxeter groups} are efficiently recorded in a \emph{Coxeter diagram} that is a graph whose vertices are labeled with the generators of the group $S=\{s_i\mid i \in [n-1]\}$ and an edge $(s_i, s_j)$ is present if $(s_is_j)^{m_{i,j}}$ is a relator with $m_{i,j} \geq 3$. If $m_{i,j}>3$ then the edge is labeled with the value $m_{i,j}$. If there is not and edge between $s_i,s_j$ then $(s_is_j)^2$ is a relator. The only other relators for Coxeter group are that $s_i^2$ for any generator $s_i \in S$. The Coxeter diagrams for type $A_{n-1}$, $\bn$, and $\dn$ Coxeter groups are given in Figure~\ref{f:coxdiag}. The generators $s_i$, for $i \in [n-1]$, are identified with the \emph{adjacent transpositions}, i.e.~permutations $(i, i+1)$ in cycle notation. The generator $s_0$ is identified with the signed permutation $[-1, 2, 3, \ldots, n]$ in window notation. The generator $s'_0$ is identified with the signed permutation $[-2, -1, 3, 4, \ldots, n]$ in window notation. The explicit presentations of each type of group $\sn, \bn,$ and $\dn$ are given below.

For type $A_{n-1}$, the relators are:
\begin{align}
    \tag{Ca1} s_{i}^2,& \text{ for } i\in[n-1];\label{CA1}\\
    \tag{Ca2} (s_{i}s_{i+1})^3,& \text{ for } i\in[n-2]; \text{ and }\label{CA2}\\
    \tag{Ca3} (s_{i}s_{j})^2,& \text{ for } i\in[n-3] \text{ and } j\in[i+2,n-1]. \label{CA3}
\end{align}
Thus $\sn \cong \langle s_1, s_2, \ldots, s_{n-1} \mid (\ref{CA1}), (\ref{CA2}), (\ref{CA3})\rangle$.

For type $\bn$, the relators are:
\begin{align}
    \tag{Cb1} s_i,& \text{ for } i\in[0,n-1]; \label{Cb1}\\
    \tag{Cb2} (s_0s_1)^4;& \label{Cb2}\\
    \tag{Cb3} (s_is_{i+1})^3,& \text{ for } i\in[n-2]; \text{ and }\label{Cb3}\\
    \tag{Cb4} (s_is_j)^2,& \text{ for } i\in[0,n-3] \text{ and } j\in[i+2,n-1].\label{Cb4}
\end{align}
Thus $\bn \cong \langle s_0, s_1, \ldots, s_{n-1} \mid (\ref{Cb1}), (\ref{Cb2}), (\ref{Cb3}), (\ref{Cb4})\rangle$.

For type $\dn$, the relators are :
\begin{align}
    \tag{Cd1} s_i^2,& \text{ for } i\in[n-1]; \label{Cd1}\\
    \tag{Cd2} s_0'^2;& \label{Cd2}\\
    \tag{Cd3} (s_0's_2)^3;& \label{Cd3}\\
    \tag{Cd4} (s_is_{i+1})^3,& \text{ for } i\in[n-2] \label{Cd4}\\
    \tag{Cd5} (s_0's_i)^2,& \text{ for } i\in\{1\}\cup[3,n-1] \text{ and } \label{Cd5}\\
    \tag{Cd6} (s_is_j)^2,& \text{ for } i\in[n-1] \text{ and } j\in[i+2,n-1]. \label{Cd6}
\end{align}
Thus $\dn \cong \langle s_0', s_1, s_2, \ldots, s_{n-1} \mid (\ref{Cd1}), (\ref{Cd2}), (\ref{Cd3}), (\ref{Cd4}), (\ref{Cd5}), (\ref{Cd6})\rangle$.

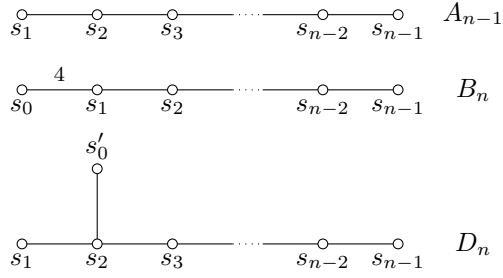
\begin{figure}\label{f:coxdiag}
    \centering
    \begin{tikzpicture}[scale=1]
        \draw (0,0) circle (0.065) node[anchor=north] {$s_1$};
        \draw (1,0) circle (0.065) node[anchor=north] {$s_2$};
        \draw (2,0) circle (0.065) node[anchor=north] {$s_3$};
        \draw (4,0) circle (0.065) node[anchor=north] {$s_{n-2}$};
        \draw (5,0) circle (0.065) node[anchor=north] {$s_{n-1}$};
        \draw (0.065,0)--(0.935,0) (1.065,0)--(1.935,0) (2.065,0)--(2.8,0) (3.2,0)--(3.935,0) (4.065,0)--(4.935,0);
        \draw[dotted] (2.8,0)--(3.2,0);
        \draw (6,0) node {$A_{n-1}$};
        \def\ra{-1}
        \draw (0,\ra) circle (0.065) node[anchor=north] {$s_0$};
        \draw (1,\ra) circle (0.065) node[anchor=north] {$s_1$};
        \draw (2,\ra) circle (0.065) node[anchor=north] {$s_2$};
        \draw (4,\ra) circle (0.065) node[anchor=north] {$s_{n-2}$};
        \draw (5,\ra) circle (0.065) node[anchor=north] {$s_{n-1}$};
        \draw (0.065,\ra)--(0.935,\ra) (1.065,\ra)--(1.935,\ra) (2.065,\ra)--(2.8,\ra) (3.2,\ra)--(3.935,\ra) (4.065,\ra)--(4.935,\ra);
        \draw[dotted] (2.8,\ra)--(3.2,\ra);
        \draw (0.5,\ra) node[anchor=south] {\footnotesize{4}};
        \draw (6,\ra) node {$\bn$};
        \def\ra{-3.05}
        \draw (1,\ra+1) circle (0.065) node[anchor=south] {$s'_0$};
        \draw (0,\ra) circle (0.065) node[anchor=north] {$s_1$};
        \draw (1,\ra) circle (0.065) node[anchor=north] {$s_2$};
        \draw (2,\ra) circle (0.065) node[anchor=north] {$s_3$};
        \draw (4,\ra) circle (0.065) node[anchor=north] {$s_{n-2}$};
        \draw (5,\ra) circle (0.065) node[anchor=north] {$s_{n-1}$};
        \draw (0.065,\ra)--(0.935,\ra) (1.065,\ra)--(1.935,\ra) (2.065,\ra)--(2.8,\ra) (3.2,\ra)--(3.935,\ra) (4.065,\ra)--(4.935,\ra) (1,\ra+0.065)--(1,\ra+0.935);
        \draw[dotted] (2.8,\ra)--(3.2,\ra);
        \draw (6,\ra) node {$\dn$};
    \end{tikzpicture}
    \caption{Coxeter diagrams of $A_{n-1}, \bn,$ and $\dn$}
\end{figure}

In the symmetric group $\sn$, for each $i\in[2,n]$, the \emph{prefix reversal} $r_i$ is identified with the permutation $[i, (i-1), \ldots, 1, (i+1), (i+2), \ldots, n]$ in one-line notation. In the hyperoctahedral group $\bn$, for each $i\in[n]$, the \emph{signed prefix reversal} $r_i$ is identified with the signed permutation $[-i, -(i-1), \ldots, -1, (i+1), (i+2), \ldots, n]$ in window notation. We overused the name for these reversal to ease notation, however, we shall make clear which is appropriate in each section.

\section{Presentation of $A_{n-1} = \sn$}\label{s:snpresentation}

The main result of this section is a presentation for the symmetric group of degree $n$ using prefix reversals as a generator set.

\begin{thm}\label{t:snpres}
    A presentation for the symmetric group of degree $n>3$ has generators $\left\{r_{2},r_{3},\ldots,r_{n}\right\}$
    and complete set of relators 
    \begin{align}
        \tag{R1} r_{k}^2 &,\quad \text{for } k \in [2,n]; \label{R1}\\
        \tag{R2} \left(r_{2}r_{3}\right)^3&; \label{R2}\\
        \tag{R3} \left(r_{2}r_{k}\right)^4&,\quad \text{for } k \in [4,n]; \label{R3}\\
        \tag{R4} r_{\ell}r_{\ell-k+2}r_{2}r_{\ell-k+2}r_{\ell} r_{k}r_{2}r_{k}&, \quad \text{for } \ell \in [4,n] \text{ and } k \in [3,\ell-1]; \label{R5}\\
        \tag{R5} r_{k}r_{k-1}r_{k+1}r_{2} r_{k+1}r_{k}r_{k+1}&, \quad \text{for } k \in [3, n-1]; \text{ and} \label{R6}\\  
        \tag{R6} (r_{k}r_{k-1})^2r_{k+1}r_{3} r_{k+1}r_{k-1}r_{k+1}&, \quad \text{for } k \in [3, n-1]. \label{R7}
    \end{align}
    
    That is $\sn \cong \left\langle r_{2}, r_{3}, \ldots r_{n} \mid (\ref{R1}), (\ref{R2}), \ldots, (\ref{R7}) \right\rangle$.
\end{thm}

Prior to proving this presentation is in fact that of $\sn$ we need a few preliminary results.

\begin{lem}\label{l:transition}
    For $n>3$ and assuming the relations (\ref{R1})-(\ref{R7}) on $\left\{r_{k} \mid k \in [2,n]\right\}$ in Theorem \ref{t:snpres}, that $s_{i-1}=r_{i}r_{2}r_{i}$ for all $i \in [2,n]$, and the standard relations on $\left\{s_{j} \mid j \in [n-1]\right\}$, then the following are true:
    \begin{align}
        r_{k} &= s_{1}(s_{2}s_{1})(s_{3}s_{2}s_{1}) \cdots (s_{k-1}s_{k-2} \cdots s_{3}s_{2}s_{1}) &\text{ for any } k \in [2,n];&\label{e:rewrite}\\
        r_{k+1}r_{k} &= s_{1}s_{2} \cdots s_{k-1}s_{k} &\text{ for any } k \in [2,n-1];&\label{e:oneapart}\\
        r_{k+2}r_{k} &= (s_{1}s_{2} \cdots s_{k} s_{k+1})(s_{1}s_{2} \cdots s_{k-1}s_{k}) &\text{ for any } k \in [2,n-2];&\label{e:twoapart}\\
        r_{k}r_{3}r_{k} &= s_{k-2}s_{k-1}s_{k-2} &\text{ for any } k \in [3,n].&\label{e:ktwok}
    \end{align}
\end{lem}
\begin{proof}
    For (\ref{e:rewrite}) we will proceed by strong induction on $k$. First note that when $k=2$, $s_{1}=r_{2}r_{2}r_{2}=r_{2}$. Also when $k=3$, $s_{1}s_{2}s_{1}=r_{2}(r_{3}r_{2}r_{3})r_{2}=r_{3}$. Suppose that (\ref{e:rewrite}) is true for all $k \leq \ell$. Thus $r_{k} = s_{1}(s_{2}s_{1})(s_{3}s_{2}s_{1}) \cdots (s_{k-1}s_{k-2} \cdots s_{3}s_{2}s_{1}) = r_{k-1} (s_{k-1}s_{k-2} \cdots s_{3}s_{2}s_{1})$. 
    
    Now consider $s_{1}(s_{2}s_{1})(s_{3}s_{2}s_{1}) \cdots (s_{\ell}s_{\ell-1} \cdots s_{3}s_{2}s_{1})$. By hypothesis we have 
    \begin{align*}
        s_{1}(s_{2}s_{1})(s_{3}s_{2}s_{1}) \cdots (s_{\ell}s_{\ell-1} \cdots s_{3}s_{2}s_{1}) &= r_{\ell} (s_{\ell})(s_{\ell-1} \cdots s_{3}s_{2}s_{1})\\
        &= r_{\ell} (r_{\ell+1}r_{2}r_{\ell+1}) (r_{\ell-1}r_{\ell})\\
        &= 
        r_{\ell+1},
    \end{align*}
    where the last equality is true by (\ref{R6}) when $k=\ell$.
    
    Furthermore, we see that  (\ref{e:oneapart}) is true as a consequence of (\ref{e:rewrite}), and (\ref{e:twoapart}) follows directly from (\ref{e:oneapart}) and that $r_{k+1}$ is its own inverse (\ref{R1}).
    
    Finally, to show (\ref{e:ktwok}) we proceed by induction on $k$. When $k=3$, $r_{3}r_{3}r_{3}=r_{3}=s_{1}s_{2}s_{1}$. Suppose that (\ref{e:ktwok}) is true for some $3 \leq k < n-1$. Now consider
    \begin{align*}
        r_{k+1}r_{3}r_{k+1} &= (r_{k+1} r_{k})r_{k} r_{3} r_{k}(r_{k} r_{k+1})\\
        &= (s_{1}s_{2} \cdots s_{k}) s_{k-2}s_{k-1}s_{k-2} (s_{k}s_{k-1} \cdots s_{1})\\
        &= (s_{1}s_{2} \cdots s_{k-1}) s_{k-2} (s_{k}s_{k-1}s_{k}) s_{k-2} (s_{k-1} s_{k-2} \cdots s_{1})\\
        &= (s_{1}s_{2} \cdots s_{k-1}) s_{k-2} (s_{k-1}s_{k}s_{k-1}) s_{k-2} (s_{k-1} s_{k-2} \cdots s_{1})\\
        &= (s_{1}s_{2} \cdots \cancel{s_{k-2}}) (\cancel{s_{k-2}} s_{k-1} \cancel{s_{k-2}}) s_{k} (\cancel{s_{k-2}} s_{k-1} \cancel{s_{k-2}}) (\cancel{s_{k-2}} s_{k-3} \cdots s_{1})\\
        &= \cancel{(s_{1}s_{2} \cdots s_{k-3})} s_{k-1} s_{k} s_{k-1} \cancel{(s_{k-3} s_{k-4} \cdots s_{1})}\\ 
        &= s_{k} s_{k+1} s_{k}.
    \end{align*}
\end{proof}

\begin{lem}\label{l:srcomm}
For $n>3$, $j \in [2,n-1]$, $i \in [1,j-1]$, $\ell \in [2,n]$, and $k \in [1, \ell-1]$, then 
\begin{align}
    s_{j} r_{i} &= r_{i} s_{j} \label{sjri} \text{ and }\\
    s_{k} r_{\ell} &= r_{\ell} s_{\ell-k} \label{skrl} 
\end{align}
follow from the identities 
\begin{align*}
        r_{i} = s_{1}(s_{2}s_{1})(s_{3}s_{2}s_{1}) \cdots (s_{i-1}s_{i-2} \cdots s_{2}s_{1}), \quad &\text{for } i \in [2,n];\\
        s_{i-1} = r_{i}r_{2}r_{i} \quad &\text{for } i \in [2,n];\\
    \end{align*}
    and the relators (\ref{CA1})-(\ref{CA3}) and (\ref{R1}).
\end{lem}
\begin{proof}
We begin with (\ref{sjri}). Suppose $j\in [2,n-1]$ and $i \in [1,j-1]$, then 
    
    \begin{align*}
        s_{j} r_{i} &= s_{j} [s_{1}(s_{2}s_{1}) \cdots (s_{i-1}s_{i-2} \cdots s_{1})] = [s_{1}(s_{2}s_{1}) \cdots (s_{i-1}s_{i-2} \cdots s_{1})] s_{j} = r_{i}s_{j}
    \end{align*}
    since $i-1 \leq j-2$ and $j - (i-1) \geq 2$.
    
To show (\ref{skrl}) we proceed by induction. First, we can verify the result is true for $\ell = 2$ and $k=1$.
       $s_{1} r_{2} = s_{1} [s_{1}] = r_{2} s_{1}.$

Assuming the result is true for some $\ell \in [2,n]$ with $k=1$, that is $s_{1} r_{\ell} = r_{\ell} s_{\ell-1}$. Now consider
    \begin{align*}
        s_{1} r_{\ell+1} &= s_{1} [s_{1}(s_{2}s_{1}) \cdots (s_{\ell}s_{\ell-1} \cdots s_{1})]\\
        &= s_{1} r_{\ell} (s_{\ell}s_{\ell-1} \cdots s_{1})\\
        &= r_{\ell} s_{\ell-1} (s_{\ell}s_{\ell-1}s_{\ell-2} \cdots s_{1})\\
        &= r_{\ell} (s_{\ell} s_{\ell-1} s_{\ell}) (s_{\ell-2} s_{\ell-3} \cdots s_{1})\\
        &= r_{\ell} (s_{\ell} s_{\ell-1} s_{\ell-2} \cdots s_{1}) s_{\ell}\\
        &= r_{\ell+1} s_{\ell}.
    \end{align*}

Finally, assuming for any $\ell \in [2,n]$ there is a $k \in [1, \ell-1]$ where $s_{k}r_{\ell} = r_{\ell}s_{\ell-k}$. Consider 
    \begin{align*}
        s_{k+1} r_{\ell} &= \boxed{s_{k+1}} [(s_{1}s_{2} \cdots s_{\ell-1})(s_{1}s_{2} \cdots s_{\ell-2}) \cdots (s_{1}s_{2}) s_{1}]\\
        &= (s_{1}s_{2} \cdots s_{k-1} \boxed{s_{k+1}} s_{k} s_{k+1} s_{k+2} \cdots s_{\ell-1}) (s_{1}s_{2} \cdots s_{\ell-2}) \cdots (s_{1}s_{2}) s_{1}\\
        &= (s_{1}s_{2} \cdots s_{k-1} (s_{k} s_{k+1} \boxed{s_{k}}) s_{k+2} \cdots s_{\ell-1}) (s_{1}s_{2} \cdots s_{\ell-2}) \cdots (s_{1}s_{2}) s_{1}\\
        &= (s_{1}s_{2} \cdots s_{\ell-1}) \boxed{s_{k}} (s_{1}s_{2} \cdots s_{\ell-2}) \cdots (s_{1}s_{2}) s_{1}\\
        &= r_{\ell}r_{\ell-1}s_{k}r_{\ell-1}\\
        &= r_{\ell}r_{\ell-1}r_{\ell-1} s_{\ell-k-1}\\
        &= r_{\ell}s_{\ell-k-1}.
    \end{align*}
\end{proof}

\noindent\emph{Proof of Theorem~\ref{t:snpres}.}
    The verification of this presentation is accomplished by using the Tietze transformations on the presentation above to recover the standard Coxeter group presentation of $\sn$, that is the presentation of $\sn$ with adjacent transpositions as generators.
    
    By T3, we can adjoin in the generators $\{s_{1}, s_{2}, \ldots, s_{n-1}\}$ and the relations 
    \begin{equation} \tag{R7} s_{i}=r_{i+1}r_{2}r_{i+1}, \text{ for } i \in [n-1].\label{R8}
    \end{equation} 
    Note that when $i=1$ the relator simplifies to $s_{1} = r_{2} r_{2} r_{2} = r_{2}$, by (\ref{R1}). 
    
    We now will adjoin the relators (\ref{CA1})-(\ref{CA3}).
    However, we must verify that these new relators are derived from (\ref{R1})-(\ref{R8}).
    Trivially, we note that for $i \in [n-1]$,
    \begin{align*}
        s_i^2 &= (r_{i+1}r_{2}r_{i+1})^2=r_{i+1}r_{2}r_{i+1}r_{i+1}r_{2}r_{i+1}=e.
    \end{align*}
    Thus (\ref{CA1}) is derivable from (\ref{R1}) and (\ref{R8}).
    Furthermore, we see that 
    \begin{align*}
        (s_{1}s_{2})^3&=(r_{2}r_{3}r_{2}r_{3})^3\\
        &=(r_{2}r_{3})^6=e.
    \end{align*}
    As well as, for $i \in [2,n-2]$, 
    \begin{align*}
        (s_{i}s_{i+1})^3&=(r_{i+1}r_{2}r_{i+1}r_{i+2}r_{2}r_{i+2})^3\\
        &=(r_{i+2}r_{3}r_{2}r_{3}r_{i+2}r_{i+2}r_{2}r_{i+2})^3\\
        &=(r_{i+2}r_{3}r_{2}r_{3}r_{2}r_{i+2})^3\\
        &=r_{i+2}r_{3}r_{2}r_{3}r_{2}r_{i+2}r_{i+2}r_{3}r_{2}r_{3}r_{2}r_{i+2}r_{i+2}r_{3}r_{2}r_{3}r_{2}r_{i+2}\\
        &=r_{i+2}(r_{3}r_{2})^6r_{i+2}=e.
    \end{align*}
    Thus (\ref{CA2}) is derivable from (\ref{R2}), (\ref{R5}), and (\ref{R8}).
    Finally, for $i \in [n-3]$ and $j \in [i+2, n-1]$, 
    \begin{align*}
        (s_{i}s_{j})^2&=(r_{i+1}r_{2}r_{i+1}r_{j+1}r_{2}r_{j+1})^2\\
        &=(r_{i+1}r_{2}r_{2}r_{i+1}r_{j+1}r_{j-i+2}r_{2}r_{j-i+2}r_{2}r_{j+1})^2\\
        &=(r_{j+1}r_{j-i+2}r_{2}r_{j-i+2}r_{2}r_{j+1})^2\\
        &=r_{j+1}r_{j-i+2}r_{2}r_{j-i+2}r_{2}r_{j+1}r_{j+1}r_{j-i+2}r_{2}r_{j-i+2}r_{2}r_{j+1}\\
        &=r_{j+1}r_{j-i+2}r_{2}r_{j-i+2}r_{2}r_{j-i+2}r_{2}r_{j-i+2}r_{2}r_{j+1}\\
        &=r_{j+1}(r_{j-i+2}r_{2})^4r_{j+1}=e.
    \end{align*}
    Note that $i-j>2$, thus $i-j+1 \in [3,n-1]$. Therefore (\ref{CA3}) is derivable from (\ref{R3}), (\ref{R5}), and (\ref{R8}).
    
    By Lemma \ref{l:transition}, we see that the relations
    \begin{equation}
        \tag{R8} r_{k} = s_{1} (s_{2}s_{1}) (s_{3}s_{2}s_{1}) \cdots (s_{k-1}s_{k-2} \cdots s_{2}s_{1}),\label{R9} 
    \end{equation}
    for $k \in [2,n]$, are derivable from (\ref{R1})-(\ref{R8}). So by T1, we may adjoin them to our presentation.
    
    We now shall verify that (\ref{R1})-(\ref{R7}) are derivable from (\ref{CA1})-(\ref{CA3}) and (\ref{R8})-(\ref{R9}). 
    
    \noindent \emph{(\ref{R1}) is derivable from (\ref{CA1})-(\ref{CA3}) and (\ref{R8})-(\ref{R9}).}
    
    We will verify by induction on $k\in[2,n]$. Note $r_{2}^2 = s_{1}^2 = e$. So we assume the result, $r_{k}^2=e$, is true for some $k$. Furthermore we can conclude that \begin{equation*}
        r_{k} = s_{1} (s_{2}s_{1}) \cdots (s_{k-1}s_{k-2} \cdots s_{1}) = (s_{1}s_{2} \cdots s_{k-1}) \cdots (s_{1}s_{2}) s_{1}.
    \end{equation*}
    Now consider 
    \begin{align*}
        r_{k+1}^2 &= r_{k} (s_{k}s_{k-1} \cdots s_{1}) r_{k} (s_{k}s_{k-1} \cdots s_{1})\\
        &= r_{k} (s_{k}\cancel{s_{k-1} \cdots s_{1}}) \left[(\cancel{s_{1}s_{2} \cdots s_{k-1}})(s_{1}s_{2} \cdots s_{k-2}) \cdots (s_{1}s_{2}) s_{1}\right] (s_{k}s_{k-1} \cdots s_{1})\\
        &= r_{k} \cancel{s_{k}} \left[(s_{1}s_{2} \cdots s_{k-2}) \cdots (s_{1}s_{2})s_{1}\right] (\cancel{s_{k}}s_{k-1} \cdots s_{1})\\
        &= \left[s_{1}(s_{2}s_{1}) \cdots (s_{k-1} \cancel{s_{k-2} \cdots s_{1}})\right] (\cancel{s_{1}s_{2} \cdots s_{k-2}}) \cdots (s_{1}s_{2}) s_{1} (s_{k-1}s_{k-2} \cdots s_{1})\\
        &= \left[s_{1}(s_{2}s_{1}) \cdots (s_{k-2}s_{k-3} \cdots s_{1}) \cancel{s_{k-1}}\right] (s_{1}s_{2} \cdots s_{k-3})\cdots \\
        &\hspace{7.4cm}\cdots (s_{1}s_{2})s_{1} (\cancel{s_{k-1}}s_{k-2}\cdots s_{1})\\
        &= \left[s_{1}(s_{2}s_{1}) \cdots (s_{k-2}\cancel{s_{k-3} \cdots s_{1}})\right] (\cancel{s_{1}s_{2} \cdots s_{k-3}}) \cdots (s_{1}s_{2})s_{1} (s_{k-2}s_{k-3}\cdots s_{1})\\
         &\qquad\vdots \\
         &=[s_{1}(s_{2}s_{1})] s_{1} (s_{2}s_{1}) = e.
    \end{align*}
    
    \noindent \emph{(\ref{R2}) is derivable from (\ref{CA1})-(\ref{CA3}) and (\ref{R8})-(\ref{R9}).}
    \begin{align*}
        (r_{2}r_{3})^3 &= (s_{1}[s_{1}s_{2}s_{1}])^3\\
        &=(s_{2}s_{1})^3 = e.
    \end{align*}
    
    \noindent \emph{(\ref{R3}) is derivable from (\ref{CA1})-(\ref{CA3}) and (\ref{R8})-(\ref{R9}).}
    
    Let $k \in [4,n]$ be arbitrary, then we have that
    \begin{align*}
        (r_{2}r_{k})^4 &= (r_{2}r_{k}r_{2}r_{k})^2\\
        &= (s_{1}s_{k-1})^2 = e.
    \end{align*}
    
    \noindent \emph{(\ref{R5}) is derivable from (\ref{CA1})-(\ref{CA3}) and (\ref{R8})-(\ref{R9}).}
    
    This has been verified in Lemma \ref{l:srcomm}.
    
    \noindent \emph{(\ref{R6}) is derivable from (\ref{CA1})-(\ref{CA3}) and (\ref{R8})-(\ref{R9}).}
    
    We will verify the result by induction upon $k \in [3,n-1]$. When $k=3$ then the associated relator is 
    \begin{align*}
        r_{3}r_{2}r_{4}r_{2}r_{4}r_{3}r_{4} &= [s_{1}s_{2}s_{1}][s_{1}][s_{3}][s_{1}s_{2}s_{1}][s_{1}s_{2}s_{1}s_{3}s_{2}s_{1}]=e.
    \end{align*}
    Assuming the result is true for some $k \in [3,n-1]$, consider $k+1$.
    \begin{align*}
        r_{k+1}r_{k}r_{k+2}r_{2}r_{k+2}r_{k+1}r_{k+2} &= (r_{k}r_{k-1}r_{k+1}r_{2}r_{k+1})r_{k+2}r_{2}r_{k+2}r_{k+1}r_{k+2}\\
        &= (s_{1}s_{2} \cdots s_{k-1}) s_{k} s_{k+1} (s_{k+1}s_{k} \cdots s_{1}) = e.
    \end{align*}
    
    \noindent \emph{(\ref{R7}) is derivable from (\ref{CA1})-(\ref{CA3}) and (\ref{R8})-(\ref{R9}).}
    
    We will verify the result by induction upon $k \in [3,n-1]$. When $k=3$ then the associated relator is
    \begin{align*}
        r_{3}r_{2}r_{3}r_{2}r_{4}r_{3}r_{4}r_{2}r_{4} &= s_{2}s_{1}(s_{1}s_{2}s_{1}s_{3}s_{2}s_{1})(s_{1}s_{2}s_{1})s_{3}\\
        &=s_{1}s_{3}s_{1}s_{3} = e.
    \end{align*}
    Assuming the result is true for some $k \in [3, n-1]$, consider $k+1$.
    \begin{align*}
        (r_{k+1}r_{k})(r_{k+1}r_{k})r_{k+2}r_{3}r_{k+2}[r_{k}r_{k+2}] &= (s_{1}s_{2} \cdots s_{k}) (s_{1}s_{2} \cdots s_{k}) s_{k}s_{k+1}s_{k}\\ &\qquad [(s_{k}s_{k-1} \cdots s_{1})(s_{k+1}s_{k} \cdots s_{1})]\\
        &= (s_{1}s_{2} \cdots s_{k})(s_{1}s_{2} \cdots s_{k-1}) \cancel{s_{k+1}}\\ &\qquad (s_{k-1}s_{k-2} \cdots s_{1})(\cancel{s_{k+1}}s_{k} \cdots s_{1})\\
        &= (s_{1}s_{2} \cdots s_{k})(s_{1}s_{2} \cdots s_{k-1}) \\ &\qquad (s_{k-1}s_{k-2} \cdots s_{1})(s_{k} s_{k-1} \cdots s_{1}) = e.
    \end{align*}
    
    Since (\ref{R1})-(\ref{R7}) are derivable from (\ref{CA1})-(\ref{CA3}) and (\ref{R8})-(\ref{R9}), then by T2 they may be removed. Furthermore, by T4 we can remove the generators $\{r_{k} \mid k \in [2,n]\}$, replace $r_{k}$ using (\ref{R9}) in (\ref{R8}), and remove the relations (\ref{R9}). Specifically with this replacement in $(\ref{R8})$ we have 
    \begin{align*}
        s_{i} &= [ (s_{1}s_{2} \cdots s_{i}) (s_{1}s_{2} \cdots s_{i-1}) \cdots (s_{1}s_{2}) s_{1}] s_{1} [s_{1}(s_{2}s_{1}) \cdots (s_{i}s_{i-1} \cdots s_{1})] \\
        &= (s_{1}s_{2} \cdots s_{i}) (s_{1}s_{2} \cdots s_{i-1}) \cdots (s_{1}s_{2}s_{3}) (s_{1}s_{2} s_{1}s_{2}s_{1}) (s_{3}s_{2}s_{1}) \cdots (s_{i}s_{i-1} \cdots s_{1})\\
        &= (s_{1}s_{2} \cdots s_{i}) (s_{1}s_{2} \cdots s_{i-1}) \cdots (s_{1}s_{2}s_{3}) (s_{2}) (s_{3}s_{2}s_{1}) \cdots (s_{i}s_{i-1} \cdots s_{1})\\
        &\qquad\vdots\\
        &= (s_{1}s_{2} \cdots s_{i}) s_{i-1} (s_{i}s_{i-1} \cdots s_{1}) = s_{i},
    \end{align*}
    which is trivial and may also be removed. Therefore, we are finally left with the presentation 
    \begin{equation*}
        \left\langle s_{1}, s_{2}, \ldots, s_{n-1} \mid (\ref{CA1}), (\ref{CA2}), (\ref{CA3}) \right\rangle,
    \end{equation*}
    which is precisely the standard Coxeter group presentation of $\sn$.
\qedsymbol

\section{Presentation of $\bn$}\label{s:bnpresentation}

Throughout this section the reversals are signed, i.e. \[r_{k} = [-k, -(k-1), \ldots, -1, k+1, k+2, \ldots, n] \in \bn.\] The main result of this section is the following presentation for $\bn$.

\begin{thm}\label{t:bnpres}
A presentation for the signed symmetric group of degree $n>3$ has generators $\left\{r_{1} ,r_{2} \ldots,r_{n} \right\}$
    and complete set of relators 
    \begin{align}
        \tag{Rb1} \left(r_{k} \right)^2 &,\quad \text{for } k \in [n]; \label{Rb1}\\
        \tag{Rb2} (r_{2} r_{3} )^6 &; \label{Rb2}\\
        \tag{Rb3} \left(r_{1} r_{k} \right)^4&,\quad \text{for } k \in [2,n]; \label{Rb3}\\
        \tag{Rb4} \left(r_{1} r_{2} r_{1} r_{k} \right)^4&,\quad \text{for } k \in [4,n]; \label{Rb4}\\
        \tag{Rb5} \left(r_{k} r_{1} r_{k} r_{2} \right)^2&,\quad \text{for } k \in [3,n]; \label{Rb5}\\
        \tag{Rb6} r_{k} r_{1} r_{2} r_{1} r_{k} r_{k+1} r_{2} r_{3} r_{2} r_{1} r_{k+1} &, \quad \text{for } k \in [2,n-1] \label{Rb6}\\
        \tag{Rb7} r_{k+1} r_{1} r_{2} r_{1} r_{k+1} r_{k-1} r_{k} r_{k+1} r_{k} &, \quad \text{for } k \in [2,n-1]; \text{ and}\label{Rb8}\\
        \tag{Rb8} r_{k} r_{1} r_{2} r_{1} r_{k} r_{\ell} r_{\ell-k+2} r_{1} r_{2} r_{1} r_{\ell-k+2} r_{\ell} &, \quad \text{for } k \in [2,n-2] \text{ and } \ell \in [k+2,n].  \label{Rb7} 
    \end{align}
    
    That is $\bn \cong \left\langle r_{1} , r_{2} , \ldots, r_{n}  \mid (\ref{Rb1}), (\ref{Rb2}), \ldots, (\ref{Rb7}) \right\rangle$.
\end{thm}

In order to prove this truly is a presentation of $\bn$ we need the following results established.

\begin{lem}\label{l:little}
For $n>3$ and assuming the relations (\ref{Rb1})-(\ref{Rb8}) on $\{r_{k} \mid k \in [n] \}$ in Theorem \ref{t:bnpres}, then 
    \[
        (r_{2}r_{3}r_{1})^3 = e.
    \]
\end{lem}
\begin{proof}
    First we consider (\ref{Rb6}) when $k=2$,
    \[
        r_{2} r_{1} r_{2} r_{1} r_{2} r_{3} r_{2} r_{3} r_{2} r_{1} r_{3} = e.
    \]
    Cyclically permuting the generators gives that 
    \[
        r_{2} r_{3} r_{1} = r_{1} r_{2} r_{1} r_{2} r_{3} r_{2} r_{3} r_{2}.
    \]
    Second we can see from (\ref{Rb5}) when $k=3$,
    \[
        r_{3} r_{1} r_{3} r_{2} r_{3} r_{1} r_{3} r_{2} = e,
    \]
    or equivalently 
    \[
        r_{3} r_{2} r_{3} r_{1} = r_{1} r_{3} r_{2} r_{3}.
    \]
    Finally,
    \begin{align*}
        (r_{2} r_{3} r_{1})^3 &= (r_{1} r_{2} r_{1} r_{2} r_{3} r_{2} r_{3} r_{2}) (r_{1} r_{2} r_{1} r_{2} r_{3} r_{2} r_{3} r_{2}) (r_{1} r_{2} r_{1} r_{2} r_{3} r_{2} r_{3} r_{2})\\
        &= r_{1} r_{2} r_{1} r_{2} r_{3} r_{2} r_{3} (r_{1} r_{2} r_{1}) r_{3} r_{2} r_{3} (r_{1} r_{2} r_{1}) r_{3} r_{2} r_{3} r_{2}\\
        &= r_{1} r_{2} r_{1} r_{2} (r_{1} r_{3} r_{2} r_{3}) r_{2} (r_{3} r_{2} r_{3} \cancel{r_{1}}) \cancel{r_{1}} r_{2} (r_{3} r_{2} r_{3} r_{1}) r_{2}\\
        &= r_{1} r_{2} r_{1} r_{2} r_{1} (r_{2}) r_{1} r_{2} = e.
    \end{align*}
\end{proof}

\begin{lem}\label{l:btransition}
For $n>3$ and assuming the relations (\ref{Rb1})-(\ref{Rb8}) on $\{r_{k} \mid k \in [n] \}$ in Theorem \ref{t:bnpres}, $s_{0}=r_{1}$, $s_{i}=r_{i+1}r_{1}r_{2}r_{1}r_{i+1}$ for $i\in[n-1]$, and (\ref{Cb1})-(\ref{Cb4}) on $\{s_{i} \mid i \in [0,n-1]\}$, then
    \begin{equation*}
        r_{k} = s_{0}(s_{1}s_{0})(s_{2}s_{1}s_{0})\cdots(s_{k-1}s_{k-2} \cdots s_{0}), \text{ for } k \in [n].
    \end{equation*}
\end{lem}
\begin{proof}
    We proceed by strong induction. By hypothesis $s_{0}=r_{1}$. By (\ref{Rb3}), $r_{2}=r_{1}(r_{2}r_{1}r_{2}r_{1}r_{2})r_{1} = s_{0}s_{1}s_{0}$. Assuming the result is true for all $j \leq k$, that is $r_{j} = s_{0}(s_{1}s_{0})(s_{2}s_{1}s_{0})\cdots(s_{j-1}s_{j-2}\cdots s_{0})$. Furthermore, $s_{j-1}s_{j-2}\cdots s_{0}=r_{j-1}r_{j}$. Now consider
    \begin{align*}
        s_{0}(s_{1}s_{0})(s_{2}s_{1}s_{0})\cdots
        (s_{k}s_{k-1}\cdots s_{0}) &= r_{k} s_{k} (s_{k-1}\cdots s_{0})\\
        &= r_{k} (r_{k+1}r_{1}r_{2}r_{1}r_{k+1}) (r_{k-1}r_{k}) \\
        &= r_{k+1},
    \end{align*}
    where the last equality follows from (\ref{Rb8}).
\end{proof}

\begin{lem}\label{l:srcommb}
For $n>3$, $i \in [1,n-2]$, $j \in [i+2,n]$, $k \in [1, n-2]$, and $\ell \in [k+1,n]$, then 
\begin{align}
    s_{j} r_{i} &= r_{i} s_{j} \label{sjrib} \text{ and }\\
    s_{k} r_{\ell} &= r_{\ell} s_{\ell-k} \label{skrlb} 
\end{align}
follow from the identities 
\begin{align*}
        r_{i} = s_{0}(s_{1}s_{0})(s_{2}s_{1}s_{0}) \cdots (s_{i-1}s_{i-2} \cdots s_{1}s_{0}), \quad &\text{for } i \in [n];\\
        s_{0} = r_{1}; \quad &\\
        s_{i-1} = r_{i}r_{1}r_{2}r_{1}r_{i} \quad &\text{for } i \in [2,n];\\
    \end{align*}
    and the relators (\ref{Cb1})-(\ref{Cb4}) and (\ref{Rb1}).
\end{lem}
\begin{proof}
We begin with (\ref{sjrib}). Suppose $i\in [1,n-2]$ and $j \in [i+2,n]$, then 
    
    \begin{align*}
        s_{j} r_{i} &= s_{j} [s_{0}(s_{1}s_{0}) \cdots (s_{i-1}s_{i-2} \cdots s_{0})] = [s_{0}(s_{1}s_{0}) \cdots (s_{i-1}s_{i-2} \cdots s_{0})] s_{j} = r_{i}s_{j}
    \end{align*}
    since $j \geq i+2$ and $j - (i-1) \geq 3$.
    
To show (\ref{skrlb}) we proceed by induction. First, we can verify the result is true for $k = 1$ and $\ell=2$.
       $s_{1} r_{2} = s_{1} [s_{0}(s_{1}s_{0})] = s_{0}s_{1}s_{0}s_{1} = r_{2} s_{1}.$

Assuming the result is true for some $\ell \in [2,n]$ with $k=1$, that is $s_{1} r_{\ell} = r_{\ell} s_{\ell-1}$. Now consider
    \begin{align*}
        s_{1} r_{\ell+1} &= s_{1} [s_{0}(s_{1}s_{0}) \cdots (s_{\ell}s_{\ell-1} \cdots s_{0})]\\
        &= s_{1} r_{\ell} (s_{\ell}s_{\ell-1} \cdots s_{0})\\
        &= r_{\ell} s_{\ell-1} (s_{\ell}s_{\ell-1}s_{\ell-2} \cdots s_{0})\\
        &= r_{\ell} (s_{\ell} s_{\ell-1} s_{\ell}) (s_{\ell-2} s_{\ell-3} \cdots s_{0})\\
        &= r_{\ell} (s_{\ell} s_{\ell-1} s_{\ell-2} \cdots s_{0}) s_{\ell}\\
        &= r_{\ell+1} s_{\ell}.
    \end{align*}

Finally, assuming for any $\ell \in [2,n]$ there is a $k \in [1, \ell-1]$ where $s_{k}r_{\ell} = r_{\ell}s_{\ell-k}$. Consider 
    \begin{align*}
        s_{k+1} r_{\ell} &= \boxed{s_{k+1}} [(s_{0}s_{1} \cdots s_{\ell-1})(s_{0}s_{1} \cdots s_{\ell-2}) \cdots (s_{0}s_{1}) s_{0}]\\
        &= (s_{0}s_{1} \cdots s_{k-1} \boxed{s_{k+1}} s_{k} s_{k+1}s_{k+2} \cdots s_{\ell-1}) (s_{0}s_{1} \cdots s_{\ell-2}) \cdots (s_{0}s_{1}) s_{0}\\
        &= (s_{0}s_{1} \cdots s_{k-1} (s_{k} s_{k+1} \boxed{s_{k}}) s_{k+2} \cdots s_{\ell-1}) (s_{0}s_{1} \cdots s_{\ell-2}) \cdots (s_{0}s_{1}) s_{0}\\
        &= (s_{0}s_{1} \cdots s_{\ell-1}) \boxed{s_{k}} (s_{0}s_{1} \cdots s_{\ell-2}) \cdots (s_{0}s_{1}) s_{0}\\
        &= r_{\ell}r_{\ell-1}s_{k}r_{\ell-1}\\
        &= r_{\ell}r_{\ell-1}r_{\ell-1} s_{\ell-k-1}\\
        &= r_{\ell}s_{\ell-k-1}.
    \end{align*}
\end{proof}

\noindent\emph{Proof of Theorem~\ref{t:bnpres}.}
    We follow the same process as Theorem \ref{t:snpres}, using Tietze transformations to recover the standard Coxeter group presentation of $B_n$.
    
    By T3, we can adjoin the generators $\{s_{1},s_{2},\ldots,s_{n-1}\}$ and the relations 
    \begin{equation}
        \tag{Rb9} s_{i-1} = r_{i} r_{1} r_{2} r_{1} r_{i} , \text{ for } i \in [2,n].\label{Rb9}
    \end{equation}
    Note that another generator, $s_{0}$, is also a standard generator, which is identified with $r_{1} $. That is, $s_{0}=r_{1} $. We will use either symbol interchangeably.
    
    We may also adjoin the relators (\ref{Cb1})-(\ref{Cb4}).
    However, we must verify that these new relators are derived from (\ref{Rb1})-(\ref{Rb9}). Note for $i \in [n-1]$,
    \begin{equation*}
        s_{i}^2 = (r_{i+1} r_{1} r_{2} r_{1} r_{i+1} )^2 = r_{i+1} r_{1} r_{2} r_{1} r_{i+1} r_{i+1} r_{1} r_{2} r_{1} r_{i+1}  = e.
    \end{equation*}
    Thus (\ref{Cb1}) is derivable from (\ref{Rb1}). We can see (\ref{Cb2}) is derivable from (\ref{Rb3}) when $k=2$,
    \begin{equation*}
        (s_{0}s_{1})^4 = (r_{1} r_{2} r_{1} r_{2} r_{1} r_{2} )^4 = e.
    \end{equation*}
    We can see (\ref{Cb3}) is derivable as well from (\ref{Rb1})-(\ref{Rb3}), (\ref{Rb5})-(\ref{Rb6}), and (\ref{Rb9}). For $i \in [n-1]$,
    \begin{align*}
        (s_{i}s_{i+1})^3 &= (r_{i+1} r_{1} r_{2} r_{1} r_{1+1} r_{i+2} r_{1} r_{2} r_{1} r_{i+2} )^3\\
        &= ([r_{i+2} r_{2} r_{3} \cancel{r_{2} r_{1} r_{i+2}}] \cancel{r_{i+2} r_{1} r_{2}} r_{1} r_{i+2})^3\\
        &= (r_{i+2} r_{2} r_{3} r_{1} r_{i+2})^3\\
        &= r_{i+2} r_{2} r_{3} r_{1} \cancel{r_{i+2}} \cancel{r_{i+2}} r_{2} r_{3} r_{1} \cancel{r_{i+2}} \cancel{r_{i+2}} r_{2} r_{3} r_{1} r_{i+2}\\
        &= r_{i+2} (r_{2} r_{3} r_{1})^3 r_{i+2} = e,
    \end{align*}
    by Lemma \ref{l:little}.
    To show (\ref{Cb4}) is derivable we break into two cases, when $i=0$ and $i\neq0$.
    When $i=0$ and $j \in [2,n-1]$,
    \begin{align*}
        (s_{0} s_{j})^2 &= (r_{1} [r_{j+1} r_{1} r_{2} r_{1} r_{j+1}])^2\\
        &= (r_{j+1} r_{1} r_{j+1} r_{1} r_{j+1} r_{2} r_{1} r_{j+1})^2\\
        &= r_{j+1} r_{1} r_{j+1} r_{1} r_{j+1} r_{2} \cancel{r_{1} r_{j+1}} \cancel{r_{j+1} r_{1}} r_{j+1} r_{1} r_{j+1} r_{2} r_{1} r_{j+1}\\
        &= r_{j+1} r_{1} (r_{j+1} r_{1} r_{j+1} r_{2})^2 r_{1} r_{j+1} = e,
    \end{align*}
    where the first equality is due to (\ref{Rb9}), the second is due to (\ref{Rb3}) when $k=j+1$, and the last is due to (\ref{Rb1}) and (\ref{Rb4}).
    When $i\neq0$ and $j \in [i+2,n-1]$,
    \begin{align*}
        (s_{i}s_{j})^2 &= (r_{i+1} r_{1} r_{2} r_{1} r_{i+1} r_{j+1} r_{1} r_{2} r_{1} r_{j+1})^2\\
        &= ([r_{j+1} r_{j-i+2} r_{1} r_{2} r_{1} r_{j-i+2} \cancel{r_{j+1}}] \cancel{r_{j+1}} r_{1} r_{2} r_{1} r_{j+1})^2\\
        &= (r_{j+1} (r_{j-i+2} r_{1} r_{2} r_{1})^2 r_{j+1})^2\\
        &= r_{j+1} (r_{j-i+2} r_{1} r_{2} r_{1})^2 \cancel{r_{j+1}} \cancel{r_{j+1}} (r_{j-i+2} r_{1} r_{2} r_{1})^2 r_{j+1}\\
        &= r_{j+1} (r_{j-i+2} r_{1} r_{2} r_{1})^4 r_{j+1} = e,
    \end{align*}
    where the first equality is due to (\ref{Rb9}), the second is due to (\ref{Rb7}) with $\ell = j+1$ and $k = i+1$, the last is due to (\ref{Rb4}) with $k=j-i+2\in[4,n]$, and use of (\ref{Rb1}).
    
    By Lemma \ref{l:btransition}, we see that the relations
    \begin{equation}
        \tag{Rb10} r_{k} = s_{0}(s_{1}s_{0})(s_{2}s_{1}s_{0})\cdots(s_{k-1}s_{k-2}\cdots s_{0}),\label{Rb10}
    \end{equation}
    for $k \in [n]$, are derivable from (\ref{Rb1})-(\ref{Rb9}). So by T1, we adjoin them to our presentation.
    
    We now shall verify that (\ref{Rb1})-(\ref{Rb7}) are derivable from (\ref{Cb1})-(\ref{Cb3}) and (\ref{Rb9})-(\ref{Rb10}).
    
    \noindent \emph{(\ref{Rb1}) is derivable from (\ref{Cb1})-(\ref{Cb4}) and (\ref{Rb9})-(\ref{Rb10}).}
    
    We will verify by induction on $k \in [n]$. First note that $r_{1}^2=s_{0}^2=e$. Assume that the result is true for some $k \in [n]$, $r_{k}^2=e$. Furthermore we can see that
    \[
        r_{k} = s_{0}(s_{1}s_{0})(s_{2}s_{1}s_{0}) \cdots (s_{k-1}s_{k-2}\cdots s_{0}) = (s_{0} s_{1} \cdots s_{k-1}) \cdots (s_{0}s_{1}s_{2})(s_{0}s_{1})s_{0}.
    \]
    Now consider
    \begin{align*}
        r_{k+1}^2 &= r_{k}(s_{k}s_{k-1}\cdots s_{0}) r_{k} (s_{k}s_{k-1}\cdots s_{0})\\
        &= r_{k} (s_{k}\cancel{s_{k-1}\cdots s_{0}}) [\cancel{(s_{0} s_{1} s_{2} \cdots s_{k-1})}  \cdots (s_{0}s_{1}s_{2})(s_{0}s_{1})s_{0}] (s_{k}s_{k-1}\cdots s_{0})\\
        &=r_{k}\cancel{s_{k}}(s_{0}s_{1}s_{2} \cdots s_{k-2}) \cdots (s_{0}s_{1}s_{2})(s_{0}s_{1})s_{0}(\cancel{s_{k}}s_{k-1}\cdots s_{0})\\
        &=\left[s_{0}(s_{1}s_{0})(s_{2}s_{1}s_{0})\cdots(s_{k-1}\cancel{s_{k-2}\cdots s_{1}s_{0}})\right](\cancel{s_{0}s_{1}s_{2} \cdots s_{k-2}}) \\
        &\qquad \cdots (s_{0}s_{1}s_{2})(s_{0}s_{1})s_{0}(s_{k-1}s_{k-2}\cdots s_{0})\\
        &=\left[s_{0}(s_{1}s_{0})(s_{2}s_{1}s_{0})\cdots(s_{k-2}s_{k-3}\cdots s_{0})\cancel{s_{k-1}}\right](s_{0}s_{1}s_{2} \cdots s_{k-3}) \\&\qquad 
        \cdots (s_{0}s_{1}s_{2})(s_{0}s_{1})s_{0}(\cancel{s_{k-1}}s_{k-2}\cdots s_{0})\\
        &\qquad\vdots\\
        &=\left[s_{0}(s_{1}s_{0})\right]s_{0}(s_{1}s_{0})=e.
    \end{align*}
    
    \noindent \emph{(\ref{Rb2}) is derivable from (\ref{Cb1})-(\ref{Cb4}) and (\ref{Rb9})-(\ref{Rb10}).}
    
    \begin{align*}
        (r_{2}r_{3})^6 &= (\cancel{s_{0}(s_{1}s_{0})}\cancel{s_{0}(s_{1}s_{0})}(s_{2}s_{1}s_{0}))^6\\
        &= (s_{2}s_{1}s_{0})^6\\
        &= s_{2}s_{1}s_{0}s_{2}s_{1}s_{0}s_{2}s_{1}s_{0}s_{2}s_{1}s_{0}s_{2}s_{1}s_{0}s_{2}s_{1}s_{0}\\
        &= s_{2}s_{1}(s_{2}s_{0})s_{1}s_{0}s_{2}s_{1}(s_{2}s_{0})s_{1}s_{0}s_{2}s_{1}(s_{2}s_{0})s_{1}s_{0}\\
        &= (s_{1}s_{2}s_{1})s_{0}s_{1}s_{0}(s_{1}s_{2}s_{1})s_{0}s_{1}s_{0}(s_{1}s_{2}s_{1})s_{0}s_{1}s_{0}\\ 
        &= s_{1}s_{2}(s_{0}s_{1}\cancel{s_{0}})s_{2}(\cancel{s_{0}}s_{1}\cancel{s_{0}})s_{2}(\cancel{s_{0}}s_{1}s_{0}s_{1})\\
        &= s_{1}(s_{0}s_{2})s_{1}s_{2}s_{1}s_{2}s_{1}s_{0}s_{1}\\
        &= s_{1}s_{0} \,e\, s_{0}s_{1} = e.
    \end{align*} 
    
    \noindent \emph{(\ref{Rb3}) is derivable from (\ref{Cb1})-(\ref{Cb4}) and (\ref{Rb9})-(\ref{Rb10}).}
    
    Let $k \in [2,n]$.
    \begin{align*}
        (r_{1}r_{k})^4 &= (r_{1}r_{k}r_{1}r_{k})^2\\
        &= \left(\cancel{s_{0}}[(\cancel{s_{0}}s_{1} \cdots s_{k-1}) \cdots (s_{0}s_{1}) s_{0})]\cancel{s_{0}}[\cancel{s_{0}}(s_{1}s_{0})\cdots (s_{k-1}s_{k-2} \cdots s_{0})]\right)^2\\
        &= (s_{1}s_{2} \cdots s_{k-1}(s_{0}s_{1} \cdots s_{k-2}) \cdots(s_{0}s_{1}s_{2}s_{3}) (s_{0}s_{1}s_{2}) s_{1}s_{0}s_{1}(s_{2}s_{1}s_{0})\\
        &\qquad (s_{3}s_{2}s_{1}s_{0})\cdots (s_{k-2}s_{k-3} \cdots s_{0})(s_{k-1}s_{k-2}\cdots s_{0}))^2\\
        &= (s_{1}s_{2} \cdots s_{k-1}(s_{0}s_{1} \cdots s_{k-2}) \cdots(s_{0}s_{1}s_{2}s_{3}) s_{0}s_{2}s_{1} \cancel{s_{2}}s_{0}\cancel{s_{2}}s_{1}s_{2}s_{0}\\
        &\qquad (s_{3}s_{2}s_{1}s_{0})\cdots (s_{k-2}s_{k-3} \cdots s_{0})(s_{k-1}s_{k-2}\cdots s_{0}))^2\\
        &= (s_{1}s_{2} \cdots s_{k-1}(s_{0}s_{1} \cdots s_{k-2}) \cdots(s_{0}s_{1}s_{2}s_{3}) s_{2}\cancel{s_{0}}s_{1} s_{0}s_{1}\cancel{s_{0}}s_{2}\\
        &\qquad (s_{3}s_{2}s_{1}s_{0})\cdots (s_{k-2}s_{k-3} \cdots s_{0})(s_{k-1}s_{k-2}\cdots s_{0}))^2\\
        &= (s_{1}s_{2} \cdots s_{k-1}(s_{0}s_{1} \cdots s_{k-2}) \cdots(s_{0}s_{1}s_{2}s_{3}s_{4})s_{0}s_{1}s_{3}s_{2}\cancel{s_{3}}s_{1}s_{0}s_{1}\cancel{s_{3}}\\
        &\qquad s_{2}s_{3}s_{1}s_{0} (s_{4}s_{3}s_{2}s_{1}s_{0})\cdots (s_{k-2}s_{k-3} \cdots s_{0})(s_{k-1}s_{k-2}\cdots s_{0}))^2\\
        &= (s_{1}s_{2} \cdots s_{k-1}(s_{0}s_{1} \cdots s_{k-2}) \cdots(s_{0}s_{1}s_{2}s_{3}s_{4})s_{3}s_{0}s_{1}s_{2}s_{1}s_{0}s_{1}\\
        &\qquad s_{2}s_{1}s_{0}s_{3} (s_{4}s_{3}s_{2}s_{1}s_{0})\cdots (s_{k-2}s_{k-3} \cdots s_{0})(s_{k-1}s_{k-2}\cdots s_{0}))^2\\
        &= (s_{1}s_{2} \cdots s_{k-1}(s_{0}s_{1} \cdots s_{k-2}) \cdots(s_{0}s_{1}s_{2}s_{3}s_{4})s_{3}s_{0}s_{2}s_{1}\cancel{s_{2}}s_{0}\cancel{s_{2}}\\
        &\qquad s_{1}s_{2}s_{0}s_{3} (s_{4}s_{3}s_{2}s_{1}s_{0})\cdots (s_{k-2}s_{k-3} \cdots s_{0})(s_{k-1}s_{k-2}\cdots s_{0}))^2\\
        &= (s_{1}s_{2} \cdots s_{k-1}(s_{0}s_{1} \cdots s_{k-2}) \cdots(s_{0}s_{1}s_{2}s_{3}s_{4})s_{3}s_{2}\cancel{s_{0}}s_{1}s_{0}\\
        &\qquad s_{1}\cancel{s_{0}}s_{2}s_{3} (s_{4}s_{3}s_{2}s_{1}s_{0})\cdots (s_{k-2}s_{k-3} \cdots s_{0})(s_{k-1}s_{k-2}\cdots s_{0}))^2\\
        &\qquad \vdots\\
        &=(s_{1}s_{2} \cdots s_{k-1} (s_{k-2}s_{k-3} \cdots s_{1}s_{0}s_{1} \cdots s_{k-3}s_{k-2}) s_{k-1}s_{k-2} \cdots s_{1} s_{0})^2\\
        &\qquad \vdots\\
        &=(s_{k-1}s_{k-2} \cdots s_{3}s_{2}s_{1}\cancel{s_{2}}s_{0}\cancel{s_{2}}s_{1}s_{0}s_{2}s_{3} \cdots s_{k-2}s_{k-1})^2\\
        &=(s_{k-1}s_{k-2} \cdots s_{2} (s_{1}s_{0})^2 s_{2}s_{3} \cdots s_{k-2}s_{k-1})^2\\
        &=s_{k-1}s_{k-2} \cdots s_{2} (s_{1}s_{0})^4 s_{2}s_{3} \cdots s_{k-2}s_{k-1}\\
        &=s_{k-1}s_{k-2} \cdots s_{2} \,e\, s_{2}s_{3} \cdots s_{k-2}s_{k-1}=e.\\
    \end{align*}
    
    \noindent \emph{(\ref{Rb4}) is derivable from (\ref{Cb1})-(\ref{Cb4}) and (\ref{Rb9})-(\ref{Rb10}).}
    
    Let $k \in [4,n]$.
    \begin{align*}
        (r_{1}r_{2}r_{1}r_{k})^4 &= (\cancel{s_{0}}[\cancel{s_{0}}(s_{1}\cancel{s_{0}})]\cancel{s_{0}}[(s_{0}s_{1} \cdots s_{k-1})(s_{0}s_{1} \cdots s_{k-2}) \cdots (s_{0}s_{1})s_{0}])^4\\
        &=([s_{1} (s_{0}s_{1} \cdots s_{k-1})(s_{0}s_{1} \cdots s_{k-2}) \cdots (s_{0}s_{1}) s_{0}]\\
        &\qquad [s_{1} s_{0} (s_{1}s_{0}) \cdots (s_{k-2}s_{k-3} \cdots s_{0})(s_{k-1}s_{k-2} \cdots s_{0})])^2\\
        &=(s_{1}(s_{0}s_{1} \cdots s_{k-1})(s_{0}s_{1} \cdots s_{k-2}) \cdots (s_{0}s_{1}s_{2}) s_{1} (s_{2}s_{1}s_{0}) \\
        &\qquad \cdots (s_{k-2} \cdots s_{1}s_{0})(s_{k-1} \cdots s_{1}s_{0}))^2\\
        &=(s_{1}(s_{0}s_{1} \cdots s_{k-1})(s_{0}s_{1} \cdots s_{k-2}) \cdots (s_{0}s_{1}s_{2}s_{3})\\
        &\qquad \cancel{s_{0}}s_{2}\cancel{s_{0}}(s_{3}s_{2}s_{1}s_{0}) \cdots (s_{k-2} \cdots s_{1}s_{0})(s_{k-1} \cdots s_{1}s_{0}))^2\\
        &\qquad \vdots\\
        &=(s_{1}(s_{0}s_{1} \cdots s_{k-3}s_{k-2}s_{k-1})s_{k-2}(s_{k-1}s_{k-2}s_{k-3} \cdots s_{1}s_{0}))^2\\
        &=(s_{1}\cancel{s_{0}s_{1} \cdots s_{k-3}} s_{k-1} \cancel{s_{k-3} \cdots s_{1}s_{0}})^2\\
        &=(s_{1}s_{k-1})^2 = e.
    \end{align*}
    
    \noindent \emph{(\ref{Rb5}) is derivable from (\ref{Cb1})-(\ref{Cb4}) and (\ref{Rb9})-(\ref{Rb10}).}
    
    Let $k \in [3,n]$.
    \begin{align*}
        (r_{k}r_{1}r_{k}r_{2})^2 &= ([(s_{0}s_{1} \cdots s_{k-1})(s_{0}s_{1} \cdots s_{k-2}) \cdots (s_{0}s_{1}s_{2})(s_{0}s_{1})s_{0}]\cancel{s_{0}}\\
        &\qquad [\cancel{s_{0}}(s_{1}s_{0})(s_{2}s_{1}s_{0})\cdots (s_{k-1}s_{k-2} \cdots s_{2}\cancel{s_{1}}\cancel{s_{0}})] \cancel{s_{0}} (\cancel{s_{1}}s_{0}))^2\\
        &=((s_{0}s_{1} \cdots s_{k-1})(s_{0}s_{1} \cdots s_{k-2}) \cdots (s_{0}s_{1}s_{2})s_{1}s_{0}s_{1}(s_{2}s_{1}s_{0})\\
        &\qquad \cdots (s_{k-2} \cdots s_{1}s_{0})s_{k-1}s_{k-2} \cdots s_{2}s_{0})^2\\
        &=((s_{0}s_{1} \cdots s_{k-1})(s_{0}s_{1} \cdots s_{k-2}) \cdots (s_{0}s_{1}s_{2}s_{3})s_{0}s_{2}s_{1}\cancel{s_{2}}s_{0}\\
        &\qquad \cancel{s_{2}}s_{1}s_{2}s_{0}(s_{3}s_{2}s_{1}s_{0}) \cdots (s_{k-2} \cdots s_{1}s_{0})s_{k-1}s_{k-2} \cdots s_{2}s_{0})^2\\
        &=((s_{0}s_{1} \cdots s_{k-1})(s_{0}s_{1} \cdots s_{k-2}) \cdots (s_{0}s_{1}s_{2}s_{3})s_{2}\cancel{s_{0}}s_{1}s_{0}s_{1}\cancel{s_{0}}s_{2}\\
        &\qquad (s_{3}s_{2}s_{1}s_{0}) \cdots (s_{k-2} \cdots s_{1}s_{0})s_{k-1}s_{k-2} \cdots s_{2}s_{0})^2\\
        &=((s_{0}s_{1} \cdots s_{k-1})(s_{0}s_{1} \cdots s_{k-2}) \cdots (s_{0}s_{1}s_{2}s_{3}s_{4})\\
        &\qquad s_{0}s_{1}s_{3}s_{2}\cancel{s_{3}}s_{1}s_{0}s_{1}\cancel{s_{3}}s_{2}s_{3}s_{1}s_{0}\\
        &\qquad (s_{4}s_{3}s_{2}s_{1}s_{0}) \cdots (s_{k-2} \cdots s_{1}s_{0})s_{k-1}s_{k-2} \cdots s_{2}s_{0})^2\\
        &=((s_{0}s_{1} \cdots s_{k-1})(s_{0}s_{1} \cdots s_{k-2}) \cdots (s_{0}s_{1}s_{2}s_{3}s_{4})\\
        &\qquad s_{3}s_{0}s_{2}s_{1}\cancel{s_{2}}s_{0}\cancel{s_{2}}s_{1}s_{2}s_{0}s_{3}\\
        &\qquad (s_{4}s_{3}s_{2}s_{1}s_{0}) \cdots (s_{k-2} \cdots s_{1}s_{0})s_{k-1}s_{k-2} \cdots s_{2}s_{0})^2\\
        &=((s_{0}s_{1} \cdots s_{k-1})(s_{0}s_{1} \cdots s_{k-2}) \cdots (s_{0}s_{1}s_{2}s_{3}s_{4})\\
        &\qquad s_{3}s_{2}\cancel{s_{0}}s_{1}s_{0}s_{1}\cancel{s_{0}}s_{2}s_{3}\\
        &\qquad (s_{4}s_{3}s_{2}s_{1}s_{0}) \cdots (s_{k-2} \cdots s_{1}s_{0})s_{k-1}s_{k-2} \cdots s_{2}s_{0})^2\\
        &\qquad \vdots \\
        &=(s_{0}s_{1} \cdots s_{k-1}(s_{k-2}s_{k-3} \cdots s_{1}s_{0}s_{1} \cdots s_{k-3}s_{k-2})s_{k-1}s_{k-2} \cdots s_{2}s_{0})^2\\
    \end{align*}
    \begin{align*}
        \textcolor{white}{(r_{k}r_{1}r_{k}r_{2})^2}
        &=(s_{0}s_{1} \cdots s_{k-4}s_{k-3}s_{k-1}s_{k-2}\cancel{s_{k-1}}s_{k-3}s_{k-4} \cdots s_{1}s_{0}\\
        &\qquad s_{1} \cdots s_{k-4}s_{k-3}\cancel{s_{k-1}}s_{k-2}s_{k-1}s_{k-3}s_{k-4} \cdots s_{2}s_{0})^2\\
        &=(s_{k-1}s_{0}s_{1} \cdots s_{k-4}s_{k-3}s_{k-2}s_{k-3}s_{k-4} \cdots s_{1}s_{0}\\
        &\qquad s_{1} \cdots s_{k-4}s_{k-3}s_{k-2}s_{k-3}s_{k-4} \cdots s_{2}s_{0}s_{k-1})^2\\
        &=(s_{k-1}s_{0}s_{1} \cdots s_{k-4}s_{k-2}s_{k-3}\cancel{s_{k-2}}s_{k-4} \cdots s_{1}s_{0}\\
        &\qquad s_{1} \cdots s_{k-4}\cancel{s_{k-2}}s_{k-3}s_{k-2}s_{k-4} \cdots s_{2}s_{0}s_{k-1})^2\\
        &=(s_{k-1}s_{k-2}s_{0}s_{1} \cdots s_{k-4}s_{k-3}s_{k-4} \cdots s_{1}s_{0}\\
        &\qquad s_{1} \cdots s_{k-4}s_{k-3}s_{k-4} \cdots s_{2}s_{0}s_{k-2}s_{k-1})^2\\
        &\qquad \vdots\\
        &=(s_{k-1}s_{k-2} \cdots s_{4}s_{0}s_{1}s_{2}s_{3}s_{2}s_{1}s_{0}s_{1}s_{2}s_{3}s_{2}s_{0}s_{4}s_{5} \cdots s_{k-2}s_{k-1})^2\\
        &=(s_{k-1}s_{k-2} \cdots s_{4}s_{0}s_{1}s_{3}s_{2}\cancel{s_{3}}s_{1}s_{0}s_{1}\cancel{s_{3}}s_{2}s_{3}s_{0}s_{4}s_{5} \cdots s_{k-2}s_{k-1})^2\\
        &=(s_{k-1}s_{k-2} \cdots s_{3}s_{0}s_{1}s_{2}s_{1}s_{0}s_{1}s_{2}s_{0}s_{3}s_{4} \cdots s_{k-2}s_{k-1})^2\\
        &=(s_{k-1}s_{k-2} \cdots s_{3}s_{0}s_{2}s_{1}s_{2}s_{0}s_{1}s_{0}s_{2}s_{3}s_{4} \cdots s_{k-2}s_{k-1})^2\\
        &=(s_{k-1}s_{k-2} \cdots s_{2}s_{0}s_{1}s_{2}s_{0}s_{1}s_{0}s_{2}s_{3} \cdots s_{k-2}s_{k-1})^2\\
        &=s_{k-1}s_{k-2} \cdots s_{2}s_{0}s_{1}(s_{2}s_{0})^2 s_{1}s_{0}s_{2}s_{3} \cdots s_{k-2}s_{k-1}\\
        &=s_{k-1}s_{k-2} \cdots s_{2}s_{0}s_{1}\,e\, s_{1}s_{0}s_{2}s_{3} \cdots s_{k-2}s_{k-1}=e.
    \end{align*}
    
    \noindent \emph{(\ref{Rb6}) is derivable from (\ref{Cb1})-(\ref{Cb4}) and (\ref{Rb9})-(\ref{Rb10}).}
    
    Let $k \in [2,n-1]$.
    \begin{align*}
        r_{k}r_{1}r_{2}r_{1}r_{k}r_{k+1}r_{2}r_{3}r_{2}r_{1}r_{k+1} &= [s_{k-1}][(s_{0}s_{1} \cdots s_{k})(s_{0}s_{1} \cdots s_{k-1}) \cdots (s_{0}s_{1})s_{0}][s_{2}\cancel{s_{1}s_{0}}]\\
        &\qquad [\cancel{s_{0}s_{1}}][s_{0}(s_{1}s_{0})(s_{2}s_{1}s_{0}) \cdots (s_{k}s_{k-1} \cdots s_{0})]\\
        &=s_{0}s_{1} \cdots s_{k-4}s_{k-3}s_{k-1}s_{k-2}s_{k-1}s_{k}(s_{0}s_{1} \cdots s_{k-1}) \\
        &\qquad \cdots (s_{0}s_{1})\cancel{s_{0}}s_{2}\cancel{s_{0}}(s_{1}s_{0})(s_{2}s_{1}s_{0}) \cdots (s_{k}s_{k-1} \cdots s_{0})\\
        &=s_{0}s_{1} \cdots s_{k-4}s_{k-3}s_{k-2}s_{k-1}s_{k-2}s_{k}(s_{0}s_{1} \cdots s_{k-1}) \\
        &\qquad \cdots (s_{0}s_{1})s_{2}(s_{1}s_{0})(s_{2}s_{1}s_{0}) \cdots (s_{k}s_{k-1} \cdots s_{0})\\
        &=(s_{0}s_{1} \cdots s_{k}) s_{k-2} (s_{0}s_{1} \cdots s_{k-1}) \cdots (s_{0}s_{1})s_{2}(s_{1}s_{0})\\
        &\qquad (s_{2}s_{1}s_{0}) \cdots (s_{k}s_{k-1} \cdots s_{0})\\
        &\qquad \vdots\\
        &=(s_{0}s_{1} \cdots s_{k})(s_{0}s_{1} \cdots s_{k-1}) \cdots (s_{0}s_{1}s_{2}s_{3})s_{1}(s_{0}s_{1}s_{2})\\
        &\qquad (s_{0}s_{1})s_{2}(s_{1}s_{0})(s_{2}s_{1}s_{0}) \cdots (s_{k}s_{k-1} \cdots s_{0})\\
        &=(s_{0}s_{1} \cdots s_{k})(s_{0}s_{1} \cdots s_{k-1}) \cdots (s_{0}s_{1}s_{2}s_{3})s_{1}s_{0}s_{1}s_{0}\\
        &\qquad s_{2}s_{1}s_{2}s_{1}s_{2}s_{0}s_{1}s_{0}(s_{3}s_{2}s_{1}s_{0}) \cdots (s_{k}s_{k-1} \cdots s_{0})\\
        &=(s_{0}s_{1} \cdots s_{k})(s_{0}s_{1} \cdots s_{k-1}) \cdots (s_{0}s_{1}s_{2}s_{3})s_{1}s_{0}s_{1}s_{0}\\
        &\qquad s_{1}s_{0}s_{1}s_{0}(s_{3}s_{2}s_{1}s_{0}) \cdots (s_{k}s_{k-1} \cdots s_{0})\\
        &=(s_{0}s_{1} \cdots s_{k})(s_{0}s_{1} \cdots s_{k-1}) \cdots (s_{0}s_{1}s_{2}s_{3}) \,e\,(s_{3}s_{2}s_{1}s_{0})\\
        &\qquad \cdots (s_{k}s_{k-1} \cdots s_{0})=e.\\
    \end{align*}
    
    \noindent \emph{(\ref{Rb8}) is derivable from (\ref{Cb1})-(\ref{Cb4}) and (\ref{Rb9})-(\ref{Rb10}).}
    
    Let $k \in [2,n-1]$.
    \begin{align*}
        r_{k+1}r_{1}r_{2}r_{1}r_{k+1}r_{k-1}r_{k}r_{k+1}r_{k} &= [\cancel{s_{k}}][\cancel{s_{k-1}s_{k-2} \cdots s_{0}}][\cancel{s_{0}s_{1} \cdots s_{k-1}}\cancel{s_{k}}]=e.
    \end{align*}
    
    \noindent \emph{(\ref{Rb7}) is derivable from (\ref{Cb1})-(\ref{Cb3}) and (\ref{Rb9})-(\ref{Rb10}).}
    
    This result is shown in Lemma \ref{l:srcommb}.
    
    Since (\ref{Rb1})-(\ref{Rb7}) are derivable from (\ref{Cb1})-(\ref{Cb4}) and (\ref{Rb9})-(\ref{Rb10}), then by T2 they may be removed. Furthermore, by T4, we can remove the generators $\{r_{k} \mid k\in [2,n]\}$, replace $r_{k}$ using (\ref{Rb10}) in (\ref{Rb9}), and remove the relations (\ref{Rb10}). Specifically with this replacement in (\ref{Rb9}) we have 
    \begin{align*}
        s_{i-1} &= [(s_{0}s_{1} \cdots s_{i-1}) \cdots (s_{0}s_{1})s_{0}] \cancel{s_{0}} [\cancel{s_{0}}(s_{1}\cancel{s_{0}})] \cancel{s_{0}}  [s_{0}(s_{1}s_{0})\cdots(s_{i-1}s_{i-2}\cdots s_{0})]\\
        &= (s_{0}s_{1} \cdots s_{i-1}) \cdots (s_{0}s_{1}s_{2}) (s_{1}) (s_{2}s_{1}s_{0})\cdots(s_{i-1}s_{i-2}\cdots s_{0})\\
        &= (s_{0}s_{1} \cdots s_{i-1}) \cdots (s_{0}s_{1}s_{2}s_{3})\cancel{s_{0}}(s_{2})\cancel{s_{0}}(s_{3}s_{2}s_{1}s_{0})\cdots(s_{i-1}s_{i-2}\cdots s_{0})\\
        &\;\;\vdots\\
        &= (s_{0}s_{1} \cdots s_{i-1})s_{i-2}(s_{i-1}s_{i-2} \cdots s_{0}) = s_{i-1},\\
    \end{align*}
    which is trivial and may also be removed. Therefore we are finally left with the presentation
    \[
        \langle s_{0},s_{1},\ldots,s_{n-1} \mid (\ref{Cb1}), (\ref{Cb2}), (\ref{Cb3}), (\ref{Cb4})\rangle,
    \]
    which is precisely the standard Coxeter presentation of $\bn$.
\qedsymbol

\section{Presentation of $\dn$}\label{s:dnpresentation}

Throughout this section the prefix reversals will be unsigned, but as elements of $\bn$, i.e. \[r_{k} = [k, (k-1), \ldots, 1, (k+1), (k+2), \ldots, n] \in \bn,\] for $k \in [2,n]$. We also define $\ol{r}_{2} = [-2, -1, 3, 4, \ldots, n] \in \bn$ to be the signed reversal of the first two elements. The main result follows, and all preliminary results needed have already been established in the prior sections.

\begin{thm}\label{t:dnpres}
A presentation of the type $D$ Coxeter group of degree $n>3$ has generators
$\left\{\ol{r}_{2} ,r_{2}, r_{3}, \ldots, r_{n}\right\}$
    and complete set of relators
    \begin{align}
        \tag{Rd1} \ol{r}_{2}^2 &; \label{Rd1}\\
        \tag{Rd2} r_{k}^2 &,\quad \text{for } k \in [2,n]; \label{Rd2}\\
        \tag{Rd3} \left(\ol{r}_{2}r_{2}\right)^2&; \label{Rd3}\\
        \tag{Rd4} \left(r_{2}r_{3}\right)^3&; \label{Rd4}\\
        \tag{Rd5} \left(r_{2}r_{k}\right)^4&, \quad \text{for } k \in [4,n]; \label{Rd5}\\
        \tag{Rd6} \left(\ol{r}_{2}r_{3}r_{2}r_{3}\right)^3&; \label{Rd6}\\  
        \tag{Rd7} \left(\ol{r}_{2}r_{k}r_{2}r_{k}\right)^2&, \quad \text{for } k \in [4, n]; \label{Rd7}\\
        \tag{Rd8} r_{k}r_{k-1}r_{k+1}r_{2} r_{k+1}r_{k}r_{k+1}&, \quad \text{for } k \in [3, n-1]; \label{Rd8}\\
        \tag{Rd9} (r_{k}r_{k-1})^2r_{k+1}r_{3} r_{k+1}r_{k-1}r_{k+1}&, \quad \text{for } k \in [3, n-1]; \text{ and} \label{Rd9}\\
        \tag{Rd10} r_{\ell}r_{\ell-k+2}r_{2}r_{\ell-k+2}r_{\ell} r_{k}r_{2}r_{k}&, \quad \text{for } \ell \in [4,n] \text{ and } k \in [3,\ell-2]. \label{Rd10}
    \end{align}
    
    That is, $\dn \cong \langle \ol{r}_{2}, r_{2}, r_{3}, \ldots, r_{n} \mid (\ref{Rd1}), (\ref{Rd2}), \ldots, (\ref{Rd10}) \rangle$.
   
\end{thm}
\begin{proof}
    We follow the same process as Theorem \ref{t:snpres} and Theorem \ref{t:bnpres}, applying Tietze transformations to recover the standard Coxeter group presentation of $D_n$.
    
    By T3, we can adjoin the generators $\{s_{1},s_{2}, \ldots, s_{n-1}\}$ and the relations
    \begin{equation}
        \tag{Rd11} s_{i-1} = r_{i}r_{2}r_{i}, \text{ for } i \in [2,n].\label{Rd11}
    \end{equation}
    Note that another generator, $s_{0}'$, is also a standard generator, which is identified with $\ol{r}_{2}$. That is, $s_{0}'=\ol{r}_{2}$. We will use either symbol interchangeably. Notice also that (\ref{Cd2}) is the same as (\ref{Rd1}).
    
    We also may adjoin the relators (\ref{Cd1}), (\ref{Cd3})-(\ref{Cd6}). 
    However, we must verify that these new relators are derived from (\ref{Rd1})-(\ref{Rd11}).
    Within the proof of Theorem \ref{t:snpres} we have already shown that (\ref{Cd1}) is derivable from (\ref{Rd2}); (\ref{Cd4}) is derivable from (\ref{Rd4}) and (\ref{Rd10}); and (\ref{Cd6}) is derivable from (\ref{Rd5}) and (\ref{Rd10}). So it only remains to show that (\ref{Cd3})and (\ref{Cd5}) are derivable from (\ref{Rd1})-(\ref{Rd11}).
    First we can see
    \begin{equation*}
        (s_{0}'s_{2})^3 = (\ol{r}_{2}r_{3}r_{2}r_{3})^3=e
    \end{equation*}
    by (\ref{Rd6}). Next we see for $i=1$ in (\ref{Cd5}) that 
    \begin{equation*}
        (s_{0}'s_{1})^2 = (\ol{r}_{2}r_{2})^2 = e
    \end{equation*}
    by (\ref{Rd3}). Finally we see for $i \in [3,n-1]$ that
    \begin{equation*}
        (s_{0}'s_{i})^2 = (\ol{r}_{2}r_{i+1}r_{2}r_{i+1})^2 = e
    \end{equation*}
    by (\ref{Rd7}) with $k=i+1$.
    
    By Lemma \ref{l:transition} we see that the relations
    \begin{equation}
        \tag{Rd12} r_{k} = s_{1} (s_{2}s_{1}) \cdots (s_{k-1}s_{k-2} \cdots s_{1})\label{Rd12}
    \end{equation}
    for $k \in [n]$, are derivable from (\ref{Rd2}), (\ref{Rd4}), (\ref{Rd5}), and (\ref{Rd8})-(\ref{Rd11}).
    So by T1, we adjoin them to our presentation.
    
    Within the proof of Theorem \ref{t:snpres} we have shown that (\ref{Rd2}), (\ref{Rd4}), (\ref{Rd5}), (\ref{Rd8})-(\ref{Rd10}) are derivable from (\ref{Cd1})-(\ref{Cd6}) and (\ref{Rd11})-(\ref{Rd12}). It remains to verify that (\ref{Rd3}), (\ref{Rd6}), and (\ref{Rd7}) are derivable from (\ref{Cd1})-(\ref{Cd6}) and (\ref{Rd11})-(\ref{Rd12}).
    
    \noindent\emph{(\ref{Rd3}) is derivable from (\ref{Cd1})-(\ref{Cd6}) and (\ref{Rd11})-(\ref{Rd12}).}
    
    \begin{equation*}
        (\ol{r}_{2}r_{2})^2 = (s_{0}'s_{1})^2 = e
    \end{equation*}
    by (\ref{Cd5}) when $i=1$.
    
    \noindent\emph{(\ref{Rd6}) is derivable from (\ref{Cd1})-(\ref{Cd6}) and (\ref{Rd11})-(\ref{Rd12}).}
    
    \begin{equation*}
        (\ol{r}_{2}r_{3}r_{2}r_{3})^3 = (s_{0}'s_{2})^3=e
    \end{equation*}
    by (\ref{Cd3}).
    
    \noindent\emph{(\ref{Rd7}) is derivable from (\ref{Cd1})-(\ref{Cd6}) and (\ref{Rd11})-(\ref{Rd12}).}
    
    Let $k \in [4,n]$.
    \begin{equation*}
        (\ol{r}_{2}r_{k}r_{2}r_{k})^2 = (s_{0}' s_{k-1})^2 = e
    \end{equation*}
    by (\ref{Cd5}) with $i=k-1$.
    
    Since (\ref{Rd2})-(\ref{Rd10}) are derivable from (\ref{Cd1})-(\ref{Cd6}) and (\ref{Rd11})-(\ref{Rd12}), then by T2 they may be removed. Furthermore, by T4, we can remove the generators $\{r_{k} \mid k \in [2,n]\}$, replace $r_{k}$ using (\ref{Rd12}) in (\ref{Rd11}), and remove the relations (\ref{Rd12}). Once again in the proof of Theorem \ref{t:snpres} we had shown that the substitution of (\ref{Rd12}) into (\ref{Rd11}) yields a trivial relation that may be removed. Therefore we are finally left with the presentation
    \begin{equation*}
        \langle s_{0}', s_{1}, s_{2}, \ldots, s_{n-1} \mid (\ref{Cd1}), (\ref{Cd2}), (\ref{Cd3}), (\ref{Cd4}), (\ref{Cd5}), (\ref{Cd6}) \rangle
    \end{equation*}
    which is precisely the standard Coxeter presentation of $D_n$.
\end{proof}

\section{Conclusion}\label{s:conclusion}
The presentations that we have provided for each of these three types of Coxeter are by no means as elegant as the standard presentations. What makes these presentations worthwhile, though, is that they describe fundamental relators within the prefix reversals. This opens up the opportunity to apply word processing techniques on the prefix-reversal presentations. At the very least, these presentations are of pedagogical use. They provide a further example of an abstract definition for these particularly ubiquitous finite groups.

The authors have also yet to see the application of prefix reversals with regard to the type $D$ Coxeter groups. It appears that many of the more intriguing results regarding prefix reversals and the ``pancake problem'' are somewhat ``simpler'' in the case of signed permutations compared to unsigned. Perhaps there is an analogous relationship with the subgroup $\dn$ to $\bn$.

One could also explore applications to the pancake problem. For example, employing the Knuth-Bendix algorithm to create a confluent rewriting system from these presentations would be of particular interest. Perhaps the use of such a rewriting system may be employed to reduce randomly generated words in prefix reversals to find some probabilistic predictions of (burnt) pancake numbers that are presently unknown. 

\section{Acknowledgments}\label{s:ack}

The authors are grateful to Cassandra Carlson and Jasmine Ward for their contributions in preparing this work.

\printbibliography

\end{document}